\documentclass[12pt,a4paper,reqno,parskip=full]{amsart}
\usepackage{amsmath}
\usepackage{amsfonts}
\usepackage{amssymb}
\usepackage{colonequals}
\usepackage{parskip}

\usepackage{tikz}
\usepackage{tikz-cd}
\usetikzlibrary{matrix,arrows,decorations.pathmorphing,decorations.markings}
\tikzstyle arrowstyle=[scale=1]
\tikzstyle directed=[postaction={decorate,decoration={markings,
    mark=at position .58 with {\arrow[arrowstyle]{to}}}}]

\begingroup
    \makeatletter
    \@for\theoremstyle:=definition,remark,plain\do{%
        \expandafter\g@addto@macro\csname th@\theoremstyle\endcsname{%
            \addtolength\thm@preskip\parskip
            }%
        }
\endgroup

\numberwithin{equation}{section}
\theoremstyle{plain}

\newtheorem{theorem}[subsection]{Theorem}
\newtheorem{proposition}[subsection]{Proposition}
\newtheorem{lemma}[subsection]{Lemma}
\newtheorem{corollary}[subsection]{Corollary}

\newtheorem{question}[subsection]{Question}
\newtheorem{remark}[subsection]{Remark}

\theoremstyle{definition}

\newtheorem{definition}[subsection]{Definition}
\newtheorem{example}[subsection]{Example}

\renewcommand{\leq}{\leqslant}
\renewcommand{\geq}{\geqslant}
\newcommand{\eps}{\varepsilon}

\def\vs{\vspace{11pt}}

\DeclareMathOperator{\Aut}{Aut}

\DeclareMathOperator{\supp}{supp}

\def\A{{\mathcal A}}
\def\C{{\mathcal C}}

\def\Z{{\mathbb Z}}
\def\E{{\mathcal E}}
\def\F{{\mathbb F}}
\def\I{{\mathcal I}}
\def\P{{\mathcal P}}
\def\R{{\mathbb R}}
\def\N{{\mathbb N}}

\def\O{{\mathcal O}}
\def\S{{\mathcal S}}
\def\U{{\mathcal U}}

\begin{document}

\title{On Invariant Schreier structures}

\author{Jan Cannizzo}

\address{Department of Mathematics and Statistics\newline
\indent University of Ottawa\newline
\indent 585 King Edward Avenue\newline
\indent Ottawa, Ontario, K1N 6N5\newline
\indent Canada\newline}

\email{jcann071@uOttawa.ca}

\thanks{I am grateful to my advisor, Vadim Kaimanovich, for his support and for helpful comments regarding the preparation of this paper. I thank the anonymous referee for suggesting a number of improvements.}


\maketitle

\begin{abstract}
Schreier graphs, which possess both a graph structure and a \emph{Schreier structure} (an edge-labeling by the generators of a group), are objects of fundamental importance in group theory and geometry. We study the Schreier structures with which unlabeled graphs may be endowed, with emphasis on structures which are \emph{invariant} in some sense (e.g.\ conjugation-invariant, or \emph{sofic}). We give proofs of a number of ``folklore'' results, such as that every regular graph of even degree admits a Schreier structure, and show that, under mild assumptions, the space of invariant Schreier structures over a given invariant graph structure is very large, in that it contains uncountably many ergodic measures. Our work is directly connected to the theory of \emph{invariant random subgroups}, a field which has recently attracted a great deal of attention.
\end{abstract}

\section{Introduction}

A \emph{Schreier graph} $\Gamma$ possesses two kinds of structures, which we will for the moment refer to as a \emph{geometric structure} and an \emph{algebraic structure}. The former is the underlying graph structure, which determines the geometry of $\Gamma$, in particular allowing one to equip $\Gamma$ with a metric. The latter is the labeling of edges of $\Gamma$ with the generators of a group $G$, which one may always assume to be the free group $\F_n\colonequals\langle a_1,\ldots,a_n\rangle$. The algebraic structure is not an arbitrary labeling: each vertex $x\in\Gamma$ must be attached to precisely one ``incoming'' and one ``outgoing'' edge labeled with a given generator $a_i$. Each such labeling, together with a choice of root, identifies $\Gamma$ as a particular subgroup of $\F_n$, and in general a given unlabeled graph may possess many---indeed, even uncountably many---distinct algebraic structures.

This paper is, broadly speaking, an investigation of the algebraic structures---which we will henceforth call \emph{Schreier structures}---with which $2n$-regular graphs may be endowed (recall that a graph is $2n$-regular if each of its vertices has degree $2n$). We are especially interested in random Schreier structures which are \emph{invariant} in some sense. To be more precise, denote by $\Lambda$ the space of Schreier graphs of $\F_n$ (which are naturally rooted graphs) and by $\Omega$ the space of rooted $2n$-regular graphs, and consider the forgetful map $f:\Lambda\to\Omega$ that sends a Schreier graph to its underlying unlabeled graph. There is an induced map $f:\P(\Lambda)\to\P(\Omega)$ from the space of probability measures on $\Lambda$ to the space of probability measures on $\Omega$, and moreover the space $\P(X)$, where $X=\Lambda$ or $\Omega$, contains several subspaces of ``nice'' measures, namely: $\C(\Lambda)$, the space of probability measures on $\Lambda$ invariant under the action of $\F_n$ by conjugation; $\I(X)$, the space of measures invariant with respect to the discrete measured equivalence relation underlying $X$; $\U(X)$, the space of \emph{unimodular} measures; and $\S(X)\subseteq\U(X)$, the space of \emph{sofic} measures (roughly speaking, those measures which admit approximations by measures supported on finite graphs).

Our results may be summarized as follows:

\begin{itemize}
	\item[i.] The map $f:\Lambda\to\Omega$ is surjective, i.e.\ every $2n$-regular graph admits a Schreier structure (Theorem~\ref{lab2}).
	\item[ii.] $\C(\Lambda)=\I(\Lambda)=\U(\Lambda)$, i.e.\ the spaces of conjugation-invariant, invariant, and unimodular measures on $\Lambda$ coincide (Theorem~\ref{inv}).
	\item[iii.] $f_*\U(\Lambda)\subseteq\U(\Omega)$, i.e.\ the image of a unimodular (equivalently, invariant) measure on $\Lambda$ is a unimodular measure on $\Omega$ (Proposition~\ref{forget}).
	\item[iv.] The induced map $f:\S(\Lambda)\to\S(\Omega)$ is surjective, i.e.\ any sofic measure on $\Omega$ can be lifted to a sofic measure on $\Lambda$ (Proposition~\ref{sofsurj}).
	\item[v.] Assuming it is nonempty, the fiber $f^{-1}(\mu)$ of invariant measures over a unimodular measure $\mu\in\U(\Omega)$ supported on \emph{rigid graphs} is very large, in that it contains an uncountable family of ergodic measures, many of which we are able to describe explicitly (Theorem~\ref{uncount}).
	\item[vi.] For a large class of groups $G$, the Dirac measure $\delta_G$ concentrated on an unlabeled Cayley graph of $G$ can be lifted to a nonatomic invariant measure on $\Lambda$ (Theorem~\ref{lab3}).
\end{itemize}

The first three of these statements are certainly known to experts, yet they might best be described as ``folklore''---though they are often used, it may be difficult (and in some cases impossible) to find explicit and general proofs in the literature. Moreover, we are able to use statement iii.\ to exhibit closed invariant subspaces of $\Lambda$ which do not support an invariant measure (see Corollary~\ref{noinv} and Example~\ref{grandfather}). The latter three statements comprise the main results of the paper. Morally speaking, they show that there exists a wealth of invariant algebraic structures sitting atop a given invariant geometric structure. This is line with and expands upon recent work by Bowen~\cite{B}, who showed that the subspace of $\I(\Lambda)$ consisting of measures supported on infinite graphs is a \emph{Poulsen simplex} (the set of extremal points, i.e.\ ergodic measures, is dense). Indeed, some of our work is inspired by his.

Via the correspondence between the Schreier graphs of a given group $G$ and the lattice of subgroups $L(G)$ of that group, an invariant Schreier structure determines an \emph{invariant random subgroup}, i.e.\ a conjugation-invariant probability measure on $L(G)$. The study of invariant random subgroups has recently attracted a great deal of attention (see, for example, \cite{AGV}, \cite{ADMN}, \cite{B}, \cite{BGK}, \cite{C}, \cite{V1}, and \cite{V2}), but much about them remains unknown. Concerning our work, we do not know whether statement iv.\ above holds in full generality, i.e.\ whether any unimodular random graph supports an invariant Schreier structure, or whether it is possible to obtain a complete description of the invariant Schreier structures which sit atop a given invariant graph structure. It would also be interesting to understand invariant Schreier structures from a more algebraic point of view. The subgroups corresponding to distinct Schreier structures on the same underlying graph, for instance, are clearly isomorphic in a strong sense, but we do not know what else can be said.

\section{The space of rooted graphs}

Consider the space $\Omega$ of (isomorphism classes of) connected rooted graphs of bounded geometry, i.e.\ the space of connected graphs $\Gamma=(\Gamma,x)$ each of which is equipped with a distinguished vertex $x$, called its \emph{root}, and for which there exists a number $d$ (whose precise value will not presently concern us) such that
\[
\max_{y\in\Gamma}\deg(y)\leq d
\]
for all $\Gamma\in\Omega$. The space $\Omega$ may naturally be realized as the projective limit
\begin{equation}\label{proj}
\Omega=\varprojlim\Omega_r,
\end{equation}
where $\Omega_r$ is the set of (isomorphism classes of) $r$-neighborhoods centered at the roots of elements of $\Omega$ and the connecting morphisms $\pi_r:\Omega_{r+1}\to\Omega_r$ are restriction maps that send an $(r+1)$-neighborhood $V$ to the $r$-neighborhood $U$ of its root. (Looking at things the other way around, $\pi_r(V)=U$ only if there exists an embedding $U\hookrightarrow V$ that sends the root of $U$ to the root of $V$.) Endowing each of the sets $\Omega_r$ with the discrete topology turns $\Omega$ into a compact Polish space. Throughout this paper, we will think of an $r$-neighborhood $U\in\Omega_r$ both as a rooted graph and as the \emph{cylinder set}
\[
U=\{(\Gamma,x)\in\Omega\mid U_r(x)\cong U\},
\]
where $U_r(x)$ denotes the $r$-neighborhood of the point $x\in\Gamma$. A finite Borel measure $\mu$ on $\Omega$ is the same thing as a family of measures $\mu_r:\Omega_r\to\R$ that satisfies
\[
\mu_r(U)=\sum_{V\in\pi_r^{-1}(U)}\mu_{r+1}(V)
\]
for all $U\in\Omega_r$ and for all $r$. We will be interested primarily in the space of \emph{$2n$-regular} rooted graphs, namely rooted graphs each of whose vertices has degree $2n$, and we will also denote this space by $\Omega$. Note that imposing regularity is, in a sense, hardly restrictive: every graph of bounded geometry $d$, for instance, can be embedded into a regular graph (e.g.\ by attaching branches of the $d$-regular tree to vertices whose degrees are less than $d$).

\section{Invariant, unimodular, and sofic measures}

As is detailed in \cite{K}, there are two notions of invariance for measures $\mu$ on $\Omega$. There is invariance in the classical sense of Feldman and Moore \cite{FM}, according to which invariance is defined with respect to the underlying discrete measured equivalence relation of $\Omega$, and there is \emph{unimodularity} in the sense of Benjamini and Schramm~\cite{BS} (see also \cite{AL}). Let us go over these notions in turn. 

Consider first the equivalence relation $\E\subset\Omega\times\Omega$ whereby $(\Gamma,x)\sim(\Delta,y)$ if and only if there exists an isomorphism $\phi:\Gamma\to\Delta$ of unrooted graphs. The \emph{left projection} $\pi_\ell:\E\to\Omega$ that sends an element of $\E$ to its first coordinate determines a \emph{left counting measure} $\widetilde{\mu}_\ell$ on $\E$ with ``differential'' $d\widetilde{\mu}_\ell=d\nu_\Gamma\,d\mu$, where $\nu_\Gamma$ is the counting measure on the equivalence class of $\Gamma$. In other words, $\widetilde{\mu}_\ell$ is defined on Borel sets $E\subseteq\E$ as
\[
\widetilde{\mu}_\ell(E)=\int\nu_\Gamma(E\cap\pi_\ell^{-1}(\Gamma))\,d\mu=\int|E\cap\pi_\ell^{-1}(\Gamma)|\,d\mu.
\]
In analogous fashion, the \emph{right projection} $\pi_r:\E\to\Omega$ that sends an element of $\E$ to its second coordinate determines a \emph{right counting measure} $\widetilde{\mu}_r$ on $\E$. We now say that the measure $\mu$ is invariant if the lift $\widetilde{\mu}_\ell$ (or $\widetilde{\mu}_r$) is invariant under the involution $\iota$ given by $(\Gamma,\Delta)\mapsto(\Delta,\Gamma)$; see the following diagram.
\[
\begin{tikzcd}[column sep=small]
(\E,\widetilde{\mu}_\ell)\arrow[leftrightarrow]{rr}{\iota}\arrow{dr}[swap]{\pi_\ell}
&&(\E,\widetilde{\mu}_r)\arrow{dl}{\pi_r}\\
&(\Lambda,\mu)
\end{tikzcd}
\]
\begin{definition}\label{inv}(Invariance)
A measure $\mu$ on $\Omega$ is \emph{invariant} if $\widetilde{\mu}_\ell=\widetilde{\mu}_r$, i.e.\ if the left and right counting measures on the equivalence relation $\E$ coincide. We denote the space of invariant measures on $\Omega$ by $\I(\Omega)$.
\end{definition}

Consider next the space $\widetilde{\Omega}$ of \emph{doubly rooted graphs}, whose elements are graphs $(\Gamma,x,y)$ (which we again assume to be connected and of bounded geometry, with the same bound $d$) with a distinguished \emph{principal root} $x$ and \emph{secondary root} $y$. The left projection $\pi_x:\widetilde{\Omega}\to\Omega$ given by $(\Gamma,x,y)\mapsto(\Gamma,x)$ determines a measure $\widetilde{\mu}_x$ on $\widetilde{\Omega}$ with differential $d\widetilde{\mu}_x=dw_\Gamma\,d\mu$, where $w_\Gamma$ is the \emph{weighted counting measure} on $\Gamma$ given by
\[
w_\Gamma(y)=|\O_y(\Aut_x(\Gamma))|,
\]
i.e.\ the mass assigned to a vertex $y\in\Gamma$ is the cardinality of its orbit under the action of the stabilizer $\Aut_x(\Gamma)\leq\Aut(\Gamma)$. Thus, $\widetilde{\mu}_x$ is defined on Borel sets $E\subseteq\widetilde{\Omega}$ as
\[
\widetilde{\mu}_x(E)=\int w_\Gamma(E\cap\pi_x^{-1}(\Gamma))\,d\mu.
\]
Here as before there is a second projection, namely the right projection $\pi_y:\widetilde{\Omega}\to\Omega$ given by $(\Gamma,x,y)\mapsto(\Gamma,y)$, which, again in analogous fashion, determines a measure $\widetilde{\mu}_y$ on $\widetilde{\Omega}$. We say that the measure $\mu$ is unimodular if the lift $\widetilde{\mu}_x$ (or $\widetilde{\mu}_y$) is invariant under the natural involution given by $(\Gamma,x,y)\mapsto(\Gamma,y,x)$.

\begin{definition}(Unimodularity)
A measure $\mu$ on $\Omega$ is \emph{unimodular} if $\widetilde{\mu}_x=\widetilde{\mu}_y$, i.e.\ if the left and right weighted counting measures on the space of doubly rooted graphs coincide. We denote the space of unimodular measures on $\Omega$ by $\U(\Omega)$.
\end{definition}

Unimodularity can also be described as follows. Let $\widetilde{\Omega}^1\subset\widetilde{\Omega}$ denote the space of doubly rooted graphs $(\Gamma,x,y)$ whose principal and secondary roots are at unit distance from one another. We present $\widetilde{\Omega}^1$ as the projective limit
\begin{equation}\label{proj2}
\widetilde{\Omega}^1=\varprojlim\widetilde{\Omega}_r^1,
\end{equation}
where $\widetilde{\Omega}_r^1$ is the set of (isomorphism classes of) $r$-neighborhoods of edges that connect the principal and secondary roots of graphs $(\Gamma,x,y)\in\widetilde{\Omega}^1$. A measure $\mu$ on the projective system (\ref{proj}) may be lifted to a measure $\widetilde{\mu}$ on (\ref{proj2}) by putting
\[
\widetilde{\mu}(U,x,y)=w_U(y)\mu(U,x),
\]
where, as above, $w_U(y)=|\O_y(\Aut_x(U))|$, and the measure $\mu$ is unimodular precisely if $\widetilde{\mu}(U,x,y)=\widetilde{\mu}(U,y,x)$ for all $(U,x,y)\in\widetilde{\Omega}_r^1$ and for all $r$.

A special subspace of the space of unimodular measures on $\Omega$ is the space of \emph{sofic measures}, denoted $\S(\Omega)$. Their origin is group theoretic and goes back to Gromov~\cite{G}, who defined \emph{sofic groups} as those groups whose Cayley graphs can be approximated by a sequence of finite graphs (we will make this precise in a moment). It was later realized that the notion of soficity, which can be formulated in terms of the weak convergence of measures, naturally generalizes to objects other than groups, such as unimodular random graphs (see, once again, \cite{AL} and \cite{BS}), invariant random Schreier graphs, and, more generally, discrete measured equivalence relations \cite{EL}. It is unknown whether all unimodular measures are sofic. In fact, this question is open even for Dirac measures concentrated on Cayley graphs (that is, it is unknown whether all groups are sofic). We refer the reader to the survey of Pestov~\cite{Pe} for more on sofic groups.

To make sense of the definition of soficity, observe that a finite $2n$-regular graph $\Gamma$ naturally determines a unimodular measure on $\Omega$, namely the finitely supported measure attained by choosing a position of the root of $\Gamma$ uniformly at random. The definition of soficity now goes as follows.

\begin{definition}\label{sofic}(Soficity)
A unimodular measure $\mu\in\U(\Omega)$ is \emph{sofic} if there exists a sequence of finite graphs $\{\Gamma_i\}_{i\in\N}$ such that $\mu_i\to\mu$ weakly, where $\mu_i$ is the unimodular measure on $\Omega$ determined by $\Gamma_i$.
\end{definition}

An important fact about the space of unimodular measures is that it is closed in the weak-$*$ topology (see, for instance, \cite{K}), which shows that the space of sofic measures is indeed contained in the space of unimodular measures.

\section{Schreier graphs and Schreier structures}

Given a countable group $G$ with generating set $\A=\{a_i\}_{i\in I}$ and a subgroup $H\leq G$, consider the natural action of $G$ on the space of cosets $G\backslash H$. This action is transitive and determines a rooted graph $(\Gamma,H)$ as follows. The vertex set of $\Gamma$ is identified with $G\backslash H$, and two vertices $Hg$ and $Hg'$ are connected with an edge directed from $Hg$ to $Hg'$ and labeled with the generator $a_i$ if and only if $Hga_i=Hg'$. The graph $\Gamma$ (which is rooted at the coset $H$) is called a \emph{Schreier graph}, and we denote by $\Lambda(G)$ the \emph{space of Schreier graphs of $G$}, which we endow with a topology in the usual way (strictly speaking, $\Lambda(G)$, like $\Omega$, consists of isomorphism classes of graphs, where isomorphisms are required to respect the root and edge-labeling). Note that Schreier graphs are necessarily $2|\A|$-regular, meaning that each of their vertices has degree $2|\A|$. Schreier graphs may have both loops, i.e.\ cycles of length one, and multi-edges, i.e.\ multiple edges that join the same pair of vertices. Note also that Schreier graphs naturally generalize Cayley graphs, which arise whenever the subgroup $H$ is normal, i.e.\ when the cosets $Hg$ correspond to the elements of a group.

We will primarily be interested in Schreier graphs of the finitely generated free group of rank $n$ with a fixed set of generators, i.e.\
\[
\F_n=\langle a_1,\ldots,a_n\rangle,
\]
which, in a certain sense, subsumes all of the other cases. Our first observation is this: Given a Schreier graph $(\Gamma,H)\in\Lambda(\F_n)$, the subgroup $H\leq\F_n$ can be recovered from $\Gamma$ in a very natural way. Namely, $H$ is precisely the fundamental group $\pi_1(\Gamma,H)$, i.e.\ the set of words read upon traversing closed paths that begin and end at the coset $H$. Note that we thereby identify $\pi_1(\Gamma,H)$ as a specific subgroup of $\F_n$ and are not interested merely in its isomorphism class. By the above discussion, it follows that $\Lambda(G)\subseteq\Lambda(\F_n)$ whenever $G$ is a group with generating set $\A=\{a_1,\ldots,a_n\}$. It also follows that we could define Schreier graphs ``abstractly,'' without appealing to the coset structure determined by a subgroup of $\F_n$. That is, we could define a Schreier graph to be a (connected and rooted) $2n$-regular graph whose edges come in $n$ different colors and are colored so that every vertex is attached to precisely one ``incoming'' edge of a given color and one ``outgoing'' edge of that color. There is a natural one-to-one correspondence between the space of Schreier graphs $\Lambda(\F_n)$ viewed in the abstract and the lattice of subgroups of $\F_n$, denoted $L(\F_n)$. Namely, every subgroup $H\in L(\F_n)$ determines a Schreier graph, and every Schreier graph $\Gamma\in\Lambda(\F_n)$ determines a subgroup of $\F_n$ (by passing to the fundamental group).

\begin{definition}(Schreier structure)
Let $\Gamma\in\Omega$ be a $2n$-regular graph. A \emph{Schreier structure} $\Sigma$ on $\Gamma$ is a labeling of its edges by the generators of the free group $\F_n=\langle a_1,\ldots,a_n\rangle$ that turns $\Gamma$ into a Schreier graph, i.e.\ a map $\Sigma:E_0(\Gamma)\to\A$, where $E_0(\Gamma)$ denotes a choice of orientation for each edge $(x,y)\in\Gamma$, such that for each $x\in\Gamma$ and each $1\leq i\leq n$, there is precisely one incoming edge labeled with $a_i$ and one outgoing edge labeled with $a_i$ attached to $x$.
\end{definition}

It is natural to ask whether any (connected and rooted) $2n$-regular graph admits a Schreier structure, i.e.\ whether the forgetful map $f:\Lambda\to\Omega$ that sends a Schreier graph to its underlying unlabeled (and undirected) graph is surjective. It is well-known that this question has a positive answer, but the literature on Schreier graphs can be a bit fuzzy on this point. A statement of the result (in various forms) is to be found, for example, in \cite{Gr}, \cite{L}, \cite{dlH}, \cite{GKN}, and \cite{GN}, the latter four of which cite one another on this question, but the only proof of the claim in these sources is the one due to Gross \cite{Gr}, who showed in 1977 that every finite $2n$-regular graph can be realized as a Schreier graph of the symmetric group (this proof is reproduced in \cite{L}). In fact, seeing that every $2n$-regular graph can be realized as a Schreier graph of $\F_n$ requires nothing but classical results from graph theory that go back much further than the aforementioned sources. Let us go over the argument here. We would like to thank Grigorchuk for pointing out to us that he too has recently written a careful proof of the fact that every $2n$-regular graph admits a Schreier structure; it appears in his survey~\cite{Gri}.

A graph in $\Omega$ possesses a Schreier structure if and only if it is \emph{$2$-factorable}. Recall that a \emph{$2$-factor} of a graph $\Gamma$ is a $2$-regular subgraph of $\Gamma$ whose vertex set coincides with that of $\Gamma$. Note that a $2$-factor needn't be connected (if it were, it would be a Hamiltonian cycle). A graph is $2$-factorable if it can be decomposed into $2$-factors whose edge sets are mutually disjoint, whence the connection with Schreier structures becomes plain: if $\Gamma$ has a Schreier structure, then the subgraph $\Gamma_i$ of $\Gamma$ consisting of those edges labeled with the generator $a_i$ is a $2$-factor, and $\Gamma=\Gamma_1\cup\ldots\cup\Gamma_n$ is a $2$-factorization of $\Gamma$. Conversely, if $\Gamma=\Gamma_1\cup\ldots\cup\Gamma_n$ is a $2$-factorization of $\Gamma$, one need only give an orientation to the components of each $\Gamma_i$ and label their edges with the generator $a_i$ to obtain a labeling of $\Gamma$. The following result was proved by Petersen \cite{P} in 1891.

\begin{theorem}\label{pet}
\rm{(Petersen)} \emph{Every finite $2n$-regular graph is $2$-factorable.}
\end{theorem}

Theorem~\ref{pet} can be proved by using the fact that a finite connected graph has an \emph{Euler tour}, i.e.\ a closed path that visits every edge exactly once, if and only if each of its vertices is of even degree. One can then split any finite $2n$-regular graph into a certain bipartite graph and apply Hall's theorem (also known as the marriage lemma) to extract a $2$-factor; by induction, one obtains a $2$-factorization (see Chapter 2.1 of \cite{D} for the full argument). By the above discussion, we have the following corollary.

\begin{corollary}\label{lab}
Every finite (connected and rooted) $2n$-regular graph admits a Schreier structure.
\end{corollary}

Passing to the infinite case is made possible via an application of the \emph{infinity lemma}, which asserts that every infinite locally finite tree contains a geodesic ray; it appears in K\"{o}nig's classical text on graph theory \cite{Ko}, first published in 1936 (see Chapter~6.2), or in Chapter~8.1 of \cite{D}.

\begin{theorem}\label{lab2}
Every (connected and rooted) $2n$-regular graph admits a Schreier structure.
\end{theorem}

\begin{proof}
Let $\Gamma$ be an infinite $2n$-regular graph (the finite case has already been taken care of by Theorem \ref{pet}). Assume that $\Gamma$ is connected, and let $x_0\in\Gamma$ be an arbitrarily chosen root. Consider $U_r$, the $r$-neighborhood centered at $x_0$, and note that the cardinality of its \emph{cut set} $\C$, i.e.\ the set of edges that connect vertices in $U_r$ to vertices not in $U_r$, is even. This follows from the equation
\[
\sum_{x\in U_r}\deg(x)=2|E(U_r)|+|\C|,
\]
given that the left hand side and the first term in the right hand side are even numbers. Consider now the graph $U_r\cup\C$. By grouping the edges in $\C$ into pairs, removing each pair from $U_r\cup\C$, and connecting the vertices in $U_r$ to which the elements of each pair were attached by a new edge, we ``close up'' the neighborhood $U_r$ and turn it into a $2n$-regular graph. By Corollary \ref{lab}, this graph admits a Schreier structure, which in turn determines a labeling of $U_r$.

We now employ the infinity lemma. Let $\Sigma_r$ denote the set of Schreier structures of $U_r$ (we have just shown that $\Sigma_r$ is nonempty), and construct a tree by regarding the elements of each $\Sigma_r$ as vertices and connecting every vertex in $\Sigma_{r+1}$ by an edge to the vertex in $\Sigma_r$ that represents the Schreier structure obtained by restricting the structure on $U_{r+1}$ to $U_r$. It follows that there exists a geodesic ray in our tree, i.e.\ an infinite sequence of Schreier structures on the neighborhoods $\{U_r\}_{r\in\N}$ each of which is an extension of the last and which exhaust $\Gamma$. This implies the claim.
\end{proof}

\section{Schreier graphs versus unlabeled graphs}

In this section, we compare Schreier graphs and unlabeled graphs, focusing on the spaces of invariant and unimodular measures on these two classes of graphs and how such measures behave under the forgetful map that sends a Schreier graph to its underlying unlabeled graph. Note that a homomorphism of Schreier graphs is a homomorphism of graphs that respects the additional structure carried by a Schreier graph, i.e.\ that preserves the root and maps one edge to another only if both edges have the same label and orientation. An important feature of Schreier graphs is that this additional structure lends them a certain rigidity which is not generally enjoyed by unlabeled graphs.

\begin{proposition}\label{rigid}
The vertex stabilizer $\Aut_x(\Gamma)\leq\Aut(\Gamma)$ of a Schreier graph $(\Gamma,x)$ is always trivial.
\end{proposition}

\begin{proof}
Let $(\Gamma,x)\in\Lambda$ be an arbitrary Schreier graph, and suppose that $\phi\in\Aut_x(\Gamma)$ is a nontrivial automorphism, so that there exist distinct points $y,z\in\Gamma$ (which are necessarily equidistanced from $x$) such that $\phi(y)=z$. If $y$ and $z$ are at unit distance from $x$, then $\phi$ obviously fixes each of them, since, by definition, the edges $(x,y)$ and $(x,z)$ have different labels. If $y$ and $z$ are at distance $r\geq1$ from $x$, then consider a geodesic $\gamma:[0,r]\to\Gamma$ that joins $x$ to $y$. Since $\phi$ is an isometry, the image $\phi_*\gamma$ is a geodesic that joins $x$ to $z$. Now let $0\leq t<r$ be a value such that $\gamma(t)=\phi_*\gamma(t)$ but $\gamma(t+1)\neq\phi_*\gamma(t+1)$ (since $y\neq z$, such a value must exist). Then $\phi$  must send $\gamma(t+1)$ to $\phi_*\gamma(t+1)$, but this is impossible, since the edges $(\gamma(t),\gamma(t+1))$ and $(\gamma(t),\phi_*\gamma(t+1))$ again have different labels.
\end{proof}

The spaces $\I(\Omega)$ and $\U(\Omega)$ are not the same. The Dirac measure concentrated on an infinite vertex-transitive nonunimodular graph (such as the \emph{grandfather graph}, first constructed by Trofimov \cite{T}) is an example of a measure that is invariant but not unimodular. Conversely, taking an invariant measure supported on rigid graphs, i.e.\ graphs whose automorphism groups are trivial, and multiplying each of these graphs by a finite nonunimodular graph (such as the segment of length two) yields a measure which is unimodular but not invariant (see \cite{K}). As we will soon show, however, the notions of invariance and unimodularity coincide for Schreier graphs, and both can be viewed in terms of a third notion: conjugation-invariance.

Consider the action of $G$ on $L(G)$ by conjugation, i.e.\ the action given by $(g,H)\mapsto gHg^{-1}$. When thought of as an action on $\Lambda(G)$, it is easily seen to be continuous, and it admits an easily visualized interpretation: Given a Schreier graph $(\Gamma,H)$ and a $g\in G$, where we assume that $g$ has a fixed presentation in terms of the generators of $G$, it is possible to read the element $g$ starting from the root $H$ (or, indeed, from any other vertex). This is accomplished by following, in the proper order, edges labeled with the generators that comprise $g$ (note that following a generator $a_i^{-1}$ is tantamount to traversing a directed edge labeled with $a_i$ in the direction opposite to which the edge is pointing). Applying the element $g$ to the graph $(\Gamma,H)$ then amounts simply to ``shifting the root" of $(\Gamma,H)$ in the way just described. That is, one begins at the vertex $H$, then follows a path corresponding to the element $g$, and then declares its endpoint to be the new root. Note that if $G$ has generators of order two, then a path corresponding to an element $g\in G$ may not be unique; however, the endpoint of any path which represents $g$ is uniquely determined by $g$.

It is interesting to ask about the existence of invariant measures with respect to the action $G\circlearrowright L(G)$. Indeed, the study of such measures, which also go under the name of \emph{invariant random subgroups}, has recently attracted a great deal of attention (see \cite{AGV}, \cite{ADMN}, \cite{B}, \cite{BGK}, \cite{C}, \cite{V1}, and \cite{V2}). Let us say that a measure on $\Lambda(G)$ or, in light of the inclusion $\Lambda(G)\hookrightarrow\Lambda(\F_n)$, on $\Lambda(\F_n)=\Lambda$, is \emph{conjugation-invariant} if it is invariant under this action. Denote the space of such measures by $\C(\Lambda)$.

\begin{theorem}\label{inv}
The spaces of invariant, unimodular, and conjugation-invariant measures on the space of Schreier graphs coincide.
\end{theorem}

\begin{proof}
Proposition~\ref{rigid} implies that $\I(\Lambda)=\U(\Lambda)$. Indeed, since the vertex stabilizer of a Schreier graph $\Gamma$ is always trivial, the spaces $\E$ and $\widetilde{\Lambda}^1$ may be identified, and the weighted counting measure $w_\Gamma$ is precisely the counting measure $\nu_\Gamma$. To see that $\C(\Lambda)=\I(\Lambda)$, it is enough to know that, by the classical theory (see Corollary~1 of \cite{FM} or Proposition~2.1 of \cite{KM}), a measure is invariant in the sense of Definition~\ref{inv} if and only if it is invariant with respect to the action of a countable group whose induced orbit equivalence relation coincides with the equivalence relation $\E\subset\Omega\times\Omega$. Since $\F_n$ is clearly such a group, it follows that $\C(\Lambda)=\I(\Lambda)=\U(\Lambda)$.
\end{proof}

Let $f:\Lambda\to\Omega$ be the forgetful map that sends a Schreier graph to its underlying unlabeled graph. Our next proposition shows that $f$ sends unimodular measures to unimodular measures.

\begin{proposition}\label{forget}
The image of a unimodular measure under $f$ is unimodular, i.e.\ $f_*\U(\Lambda)\subseteq\U(\Omega)$.
\end{proposition}

\begin{proof}
Lift $\mu$ to $\widetilde{\Lambda}^1$, and consider the map $\widetilde{f}:\widetilde{\Lambda}_r^1\to\widetilde{\Omega}_r^1$ that sends a neighborhood $(U,x,y)\in\widetilde{\Lambda}^1_r$ to its underlying unlabeled neighborhood. It is easy to see that both $f$ and $\widetilde{f}$ extend to homomorphisms of projective systems and therefore that $\nu:=f_*\mu$ and $\widetilde{\nu}:=\widetilde{f}_*\widetilde{\mu}$ are measures. We thus have a diagram
\begin{equation}
\begin{tikzcd}
(\widetilde{\Lambda}_r^1,\widetilde{\mu})\arrow{r}{\widetilde{f}}\arrow{d}[swap]{\pi}
&(\widetilde{\Omega}_r^1,\widetilde{\nu})\arrow{d}{\pi}\\
(\Lambda_r,\mu)\arrow{r}{f}
&(\Omega_r,\nu)
\end{tikzcd}
\end{equation}
for each $r$, where $\Lambda_r$ and $\Omega_r$ are the images of $\widetilde{\Lambda}_r^1$ and $\widetilde{\Omega}_r^1$, respectively, under the natural projection $(U,x,y)\mapsto(U,x)$. To see that the measure $\widetilde{\nu}$ satisfies the unimodularity condition, note that for any $(U,x,y)\in\widetilde{\Omega}_r^1$, there is a one-to-one correspondence between the preimages $\widetilde{f}^{-1}(U,x,y)$ and $\widetilde{f}^{-1}(U,y,x)$, which is given simply by exchanging the principal and secondary roots of the distinguished edges of neighborhoods in $\widetilde{\Lambda}_r^1$. (This correspondence is one-to-one by Proposition~\ref{rigid}.) It is now straightforward that, since the measure $\widetilde{\mu}$ is unimodular, the aforementioned preimages have the same mass and therefore that $\widetilde{\nu}(U,x,y)=\widetilde{\nu}(U,y,x)$.

It remains to check that $\widetilde{\nu}$ is in fact the lift of $\nu$. To see this, note that, again by Proposition~\ref{rigid},
\[
|\widetilde{f}^{-1}(U,x,y)|=w_U(y)|f^{-1}(U,x)|.
\]
Moreover, we have
\[
\pi_*\widetilde{f}^{-1}(U,x,y)=f^{-1}(U,x).
\]
A bit of diagram chasing now yields the result. Starting from the upper right hand corner of our diagram, we have
\begin{align*}
\widetilde{\nu}(U,x,y)&=\widetilde{\mu}(\widetilde{f}^{-1}(U,x,y))\\
&=\frac{1}{w_U(y)}\mu(\pi_*\widetilde{f}^{-1}(U,x,y))\\
&=\frac{1}{w_U(y)}\mu(f^{-1}(U,x))\\
&=\frac{1}{w_U(y)}\nu(U,x),
\end{align*}
so that $\nu(U,x)=w_U(y)\widetilde{\nu}(U,x,y)$, as desired.
\end{proof}
\begin{remark}
As shown in~\cite{K}, this implies that an invariant measure on the space of Schreier graphs is supported on graphs which are unimodular almost surely. A result which is similar in spirit was recently attained by Biringer and Tamuz~\cite{BT}, who showed that a conjugation-invariant measure on the lattice of subgroups of a unimodular group is supported on subgroups which are unimodular almost surely.
\end{remark}
An interesting consequence of Proposition~\ref{forget} is that it allows one to exhibit closed invariant subspaces of $\Lambda$ which do not admit an invariant measure.

\begin{corollary}\label{noinv}
Let $\Gamma\in\Omega$ be an infinite vertex-transitive nonunimodular graph. Then $f^{-1}(\Gamma)$, the space of Schreier structures over $\Gamma$, is a closed invariant subspace of $\Lambda$ which does not support an invariant measure.
\end{corollary}

\begin{proof}
Let $X\colonequals f^{-1}(\Gamma)$. It is easy to see that $X$ is closed and invariant, as the equivalence class of $\Gamma$ in the space of rooted graphs (that is, the set of rerootings of $\Gamma$ up to isomorphism) consists of a single point. Suppose that $\mu$ is an invariant (equivalently, unimodular) measure supported on $X$. Then its image $f_*\mu$ is the Dirac measure on $\Gamma$. But this is a nonunimodular measure, contradicting Proposition~\ref{forget}.
\end{proof}

\begin{remark}
More generally, Corollary~\ref{noinv} can be applied to nonunimodular graphs whose equivalence classes are finite.
\end{remark}

\begin{example}\label{grandfather}
Given a rooted $d$-regular tree $T$, where $d\geq3$, together with a \emph{boundary point} $\omega\in\partial T$, one constructs the grandfather graph $\Gamma$ of Trofimov \cite{T} as follows: note first that the boundary point $\omega$ allows one to assign an orientation to each edge of $T$, namely the orientation that ``points to $\omega$,'' i.e.\ given an edge $(x,y)$, there is a unique geodesic ray $\gamma:\Z_{\geq0}\to T$ beginning either at $x$ or at $y$ and such that $\lim_{t\to\infty}\gamma(t)=\omega$, and it is the orientation of this ray that determines the orientation of $(x,y)$. Next, connect each vertex $x\in T$ to its grandfather, namely the vertex one arrives at by moving two steps towards $\omega$ with respect to the orientation just defined. The result is a $(d^2-d+2)$-regular vertex-transitive nonunimodular graph, and moreover it is not difficult to see that $X\colonequals f^{-1}(\Gamma)$ is a large (uncountable) space (e.g.\ see Theorem~\ref{lab3} below). By Corollary~\ref{noinv}, $X$ does not support an invariant measure.
\end{example}

\section{Invariant Schreier structures over unlabeled graphs}

It would be interesting to fully understand the relationship between unimodular measures on $\Lambda$ and unimodular measures on $\Omega$. We do not know, for instance, whether the induced map $f:\U(\Lambda)\to\U(\Omega)$ is surjective, i.e.\ whether, given a unimodular measure $\nu$ on the space of rooted graphs, there always exists a unimodular measure $\mu$ on the space of Schreier graphs such that $f_*\mu=\nu$. Something quite close to this statement, however, is indeed true; namely, the induced map between the spaces of sofic measures on $\Lambda$ and $\Omega$ is surjective (note that this map is well-defined, as applying the forgetful map $f$ to a sofic approximation of a measure $\mu\in\S(\Lambda)$ yields a sofic approximation of the measure $f_*\mu$).

\begin{proposition}\label{sofsurj}
The induced map $f:\S(\Lambda)\to\S(\Omega)$ is surjective.
\end{proposition}

\begin{proof}
Let $\mu\in\S(\Omega)$ be a sofic measure and $\{\Gamma_i\}_{i\in\N}$ a sofic approximation of $\mu$ consisting of $2n$-regular graphs. By Theorem~\ref{lab2}, each $\Gamma_i$ may be endowed with a Schreier structure $\Sigma_i$. We thus obtain a sequence of measures $\nu_i\in\S(\Lambda)$, namely those arising from the graphs $(\Gamma_i,\Sigma_i)$. By compactness, this sequence has a convergent subsequence whose limit measure $\nu$ is obviously sofic and, moreover, must map to $\mu$ under $f$.
\end{proof}

A further natural question of interest is to describe the fiber of invariant measures $f^{-1}(\mu)$ over a given unimodular measure $\mu\in\U(\Omega)$. Although we are unable to answer this question in full generality, we are able to show that, under mild assumptions, this fiber is very large, in that it contains uncountably many ergodic measures. Invariant Schreier structures, in other words, are not ``trivial decorations'' but themselves possess a rich structure. The aforementioned mild assumption is rigidity. To be more precise, a graph is said to be \emph{rigid} if its automorphism group is trivial, and we require that our unimodular measure $\mu$ be supported on rigid graphs. Such an assumption is not very restrictive and, indeed, even natural, as essentially all known examples of invariant measures on the space of rooted graphs (such as random Galton-Watson trees---see \cite{LPP}---or their horospheric products \cite{KS}) are supported on rigid graphs.

In proving the following results, we will understand an \emph{$a_i$-cycle} to be any graph obtained by choosing a vertex $x$ in a Schreier graph and, with $x$ as our starting point, ``following the generator $a_i$'' in both directions as far as one can go. An $a_i$-cycle is thus always isomorphic to the Cayley graph of a cyclic group with generating set $\A=\{a_i\}$. A fundamental operation on $a_i$-cycles for us will be \emph{reversal}; that is, given an $a_i$-cycle, one may always reverse its orientation by applying the formal inversion $a_i\mapsto a_i^{-1}$ to its labels. Note that this operation does not destroy the Schreier structure of a graph (although it may well yield a new Schreier structure). We first establish a lemma.

\begin{lemma}\label{size}
Let $\Gamma$ be a Schreier graph whose underlying unlabeled graph is rigid, and let $a_i$ be a fixed generator of $\F_n$. Then the space $X$ of Schreier graphs obtained by independently reversing the orientations of $a_i$-cycles in $\Gamma$ or keeping their orientations fixed is either finite or uncountable.
\end{lemma}

\begin{proof}
Let $\{C_j\}_{j\in J}$, where $J\subseteq\N$, be an enumeration of the $a_i$-cycles in $\Gamma$, and consider the space $\{0,1\}^J$. For each $\omega=(\omega_j)_{j\in J}\in\{0,1\}^J$, denote by $\Gamma_\omega$ the Schreier graph obtained from $\Gamma$ by fixing the orientation of the $a_i$-cycle $C_j$ if $\omega_j=0$ and reversing it if $\omega_j=1$. The space $X$ is in one-to-one correspondence with $\{0,1\}^J$: on the one hand, each $\Gamma\in X$ can be realized as some $\Gamma_\omega$ (by recording the orientation of each of its $a_i$-cycles), and if $\Gamma$ and $\Delta$ are distinct elements of $X$, then clearly $\Gamma_\omega\neq\Delta_{\omega'}$. Conversely, if $\omega\neq\omega'$, then $\Gamma_\omega\neq\Gamma_{\omega'}$. Indeed, let $j\in J$ be an index for which $\omega_j\neq\omega_j'$. Then if $\Gamma_\omega$ and $\Gamma_{\omega'}$ have isomorphic Schreier structures, there must exist a nontrivial automorphism $\phi:\Gamma\to\Gamma$ of the underlying unlabeled graph (as the identity map preserves the orientation of $C_j$), which contradicts the fact that our Schreier graph is rigid. We thus find that $X$ is finite if and only if $J$ is finite and uncountable otherwise.
\end{proof}

\begin{theorem}\label{uncount}
Let $\mu\in\U(\Omega)$ be a nonatomic ergodic measure supported on rigid graphs. Then provided it is nonempty, the fiber $f^{-1}(\mu)$ of invariant measures over $\mu$ contains uncountably many ergodic measures.
\end{theorem}

\begin{proof}
Let $\nu\in f^{-1}(\mu)$ be a lift of $\mu$ to a (necessarily nonatomic) invariant measure on $\Lambda$; assume, moreover, that $\nu$ is also ergodic. Put $X\colonequals\supp(\nu)$, and let $p\in(0,1)$ be a fixed probability. By the pigeonhole principle, there must exist a generator $a_i$ of $\F_n$ such that a $\nu$-random Schreier graph $\Gamma$ contains infinitely many $a_i$-cycles with positive probability, since otherwise $\Gamma$ would be finite almost surely and $\nu$ would be an atomic measure. By ergodicity, it must in fact be the case that almost every $\Gamma\in X$ contains infinitely many $a_i$-cycles. For each Schreier graph $\Gamma\in X$, denote by $\nu_{\Gamma,p}$ the Bernoulli measure over $f(\Gamma)$---the underlying unlabeled graph---obtained by independently reversing the orientation of each $a_i$-cycle (for our chosen index $i$) of $\Gamma$ with probability $p$. By Lemma~\ref{size}, these measures are nonatomic. Denote by $\nu_p$ the measure obtained by integrating the measures $\nu_{\Gamma,p}$ against the base measure $\nu$.

The measure $\nu_p$ can be described explicitly as follows. Let $U\in\Lambda_r$ be a cylinder set for which $\nu(U)>0$. The graph $U$ has an obvious ``cycle decomposition,'' namely the $2$-factorization that comes from its Schreier structure; independently reversing (with probability $p$) the orientations of the $a_i$-cycles in this factorization yields a (conditional) Bernoulli measure on the set of neighborhoods $U_1,\ldots,U_k$ with the same cycle decomposition as $U$. Since reversing the orientation of a cycle in $U$ may yield a neighborhood isomorphic to $U$, we must quotient isomorphic neighborhoods $U_i\cong U_j$. Doing this for all $U\in\Lambda_r$ determines the measures that $\nu_p$ assigns to cylinder sets and also makes plain that, if $p\neq q$, then $\nu_p\neq\nu_q$.

It is not difficult to see that $\nu_p$ is invariant; indeed, passing to the space $\widetilde{\Lambda}^1$ of doubly rooted graphs, it is obvious that, for a given doubly rooted neighborhood $(U,x,y)\in\widetilde{\Lambda}_r^1$, we have $\nu_p(U,x,y)=\nu_p(U,y,x)$, since the cycle decomposition of a neighborhood is independent of a choice of basepoint(s). Moreover, the measure $\nu_p$ is ergodic: Put $\widetilde{X}\colonequals\supp(\nu_p)$ and denote by $\pi:\widetilde{X}\to X$ the obvious projection of $\widetilde{X}$ onto $X$, and suppose that $A\subset\widetilde{X}$ is a nontrivial invariant set. Assume for the moment that $A$ is a union of cylinder sets. Then there exists a cylinder set $U\subset\widetilde{X}$ such that $A\cap U=\emptyset$, and by ergodicity of the measure $\nu$, for every $\Gamma=(\Gamma,H)\in A$ there exist infinitely many $g\in\F_n$ (corresponding to infinitely many distinct positions of the root of $\Gamma$) such that $(\Gamma,gHg^{-1})\in\pi(U)$. On the other hand, the set of Schreier graphs $(\Gamma,H)$ such that $(\Gamma,gHg^{-1})\notin U$ for all $g\in\F_n$ is a null set with respect to any conditional measure $\nu_{\Gamma,p}$ and hence a null set with respect to $\nu_p$. It follows that $\nu_p(A)=0$, a contradiction. Since $A$ can be approximated to arbitrary accuracy by unions of cylinder sets (i.e.\ for any $\eps>0$, there exists a union of cylinder sets $A_\eps$ with $\nu_p(A\bigtriangleup A_\eps)<\eps$), we find that $A$ must be trivial.
\end{proof}

Let us consider highly nonrigid graphs as well. The following theorem shows that in the case when $\nu$ is the Dirac measure concentrated on an unlabeled Cayley graph, it can very often be lifted to a nonatomic measure in $\I(\Lambda)=\U(\Lambda)$.

\begin{theorem}\label{lab3}
Let $G$ be an infinite noncyclic group, and suppose $\A=\{a_1,\ldots,a_n\}$ is a generating set for $G$ such that none of the elements $a_ia_j\in G$, for distinct $1\leq i,j\leq n$, is of order two. Then there exists a nonatomic measure $\mu\in\I(\Lambda)$ such that $f_*\mu=\delta_G$, where $\delta_G$ is the Dirac measure concentrated on an unlabeled Cayley graph of $G$.
\end{theorem}

\begin{proof}
Assume, without loss of generality, that $\A$ does not contain the identity, and let $a_i\in\A$ be a generator such that $G$ (which we think of as the Cayley graph determined by $\A$) contains infinitely many $a_i$-cycles. Since $G$ is infinite and noncyclic, such an $a_i$ must exist. Now let $a_j\in\A$ be a generator distinct from $a_i$ (such a generator must again exist, since otherwise $G\cong\Z$), and put $a_{n+1}\colonequals a_ia_j$. Let $\A_0=\A\cup\{a_{n+1}\}$, and let $G_0$ be the Cayley graph of $G$ determined by our new generating set.

Consider now the space $X\subset\Lambda(\F_{n+1})$ obtained from $G_0$ by independently reversing the orientation of each $a_{n+1}$-cycle contained in $G_0$ or leaving it the same. We claim that the space $X$ is uncountable: Let $\Gamma,\Gamma'\in X$ be two relabelings of $G_0$ such that $\Gamma$ keeps the orientation of a particular $a_{n+1}$-cycle $C$ the same whereas $\Gamma'$ reverses it.

Next, choose a vertex $x$ in $C\subset \Gamma$, and let $y$ denote the vertex reached upon traversing the outgoing edge labeled with $a_{n+1}$ attached to $x$. Let $\gamma,\gamma':[0,r]\to G$ be geodesics in $G$ (and not in $G_0$) that connect the origin to $x$ and to $y$, respectively (note that $\gamma$ and $\gamma'$ may be empty), and denote by $H$ and $H'$ the subgroups corresponding to the graphs $\Gamma$ and $\Gamma'$, respectively. Then
\[
w(\gamma)a_{n+1}w(\gamma')^{-1}\equalscolon h\in H,
\]
where $w(\gamma)$ and $w(\gamma')$ are the words read upon traversing $\gamma$ and $\gamma'$. But it is not difficult to see that $h\in H'$ if and only if $a_{n+1}=a_{n+1}^{-1}$, i.e.\ if and only if $a_{n+1}$ has order two, a contradiction (see Figure~1). It follows that if $\Gamma,\Gamma'\in X$ assign different orientations to a particular $a_{n+1}$-cycle, then they represent distinct subgroups of $\F_{n+1}$. On the other hand, the number of ways to assign orientations to the $a_{n+1}$-cycles in $G_0$ is clearly uncountable. Therefore, $X$ is uncountable.

By choosing to reverse the orientations of $a_{n+1}$-cycles independently of one another with a fixed probability $p\in(0,1)$, we obtain a measure $\mu$ whose support is $X$ and which, in light of the fact that $X$ is uncountable, is nonatomic. The measure $\mu$ is ergodic by the same argument given in Theorem~\ref{uncount}.
\end{proof}
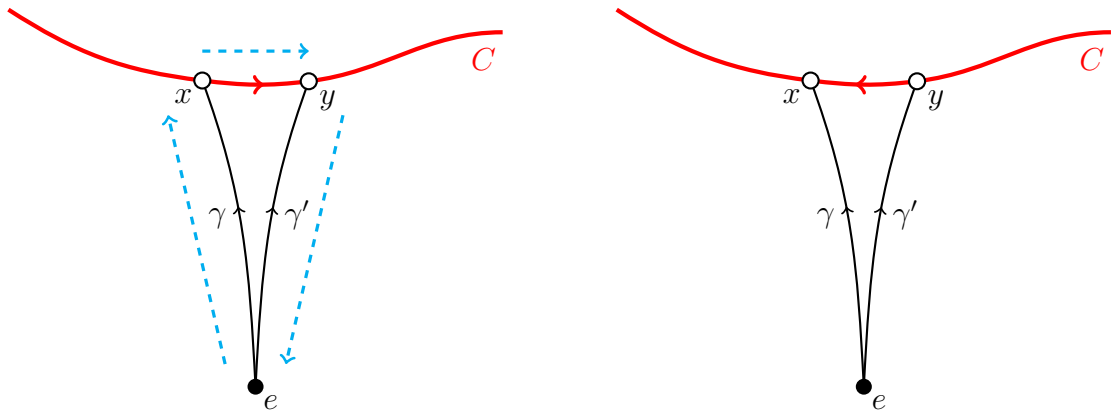
\begin{figure}
\centering
\begin{tikzpicture}[scale=1,root node/.style={circle,thick,draw=none,fill,black,inner sep=0pt,minimum size=6pt},normal node/.style={circle,thick,draw=black,fill=white,inner sep=0pt,minimum size=6pt}]


\draw[thick,directed](4,1) to[out=93,in=-70](3.3,5.06);
\draw[thick,directed](4,1) to[out=87,in=-110](4.7,5.05);

\draw[red,ultra thick](0.75,6) to[out=-33,in=173](3.3,5.06);
\draw[red,ultra thick,directed](3.3,5.06) to[out=-7,in=-173](4.7,5.05);
\draw[red,ultra thick](4.7,5.05) to[out=7,in=180](7.25,5.7);

\draw[->,cyan,dashed,very thick](3.6,1.3)--(2.85,4.6);
\draw[->,cyan,dashed,very thick](5.15,4.6)--(4.4,1.3);
\draw[->,cyan,dashed,very thick](3.3,5.45)--(4.7,5.45);

\node[root node] at (4,1){};
\node[normal node] at (3.3,5.06){};
\node[normal node] at (4.7,5.05){};

\node at (4.2,0.8){$e$};
\node at (3.5,3.23){$\gamma$};
\node at (4.55,3.28){$\gamma'$};
\node at (3.05,4.85){$x$};
\node at (4.95,4.8){$y$};
\node[red] at (7,5.35){$C$};


\draw[thick,directed](12,1) to[out=93,in=-70](11.3,5.06);
\draw[thick,directed](12,1) to[out=87,in=-110](12.7,5.05);

\draw[red,ultra thick](8.75,6) to[out=-33,in=173](11.3,5.06);
\draw[red,ultra thick,directed](12.7,5.05) to[out=-173,in=-7](11.3,5.06);
\draw[red,ultra thick](12.7,5.05) to[out=7,in=180](15.25,5.7);

\node[root node] at (12,1){};
\node[normal node] at (11.3,5.06){};
\node[normal node] at (12.7,5.05){};

\node at (12.2,0.8){$e$};
\node at (11.5,3.23){$\gamma$};
\node at (12.55,3.28){$\gamma'$};
\node at (11.05,4.85){$x$};
\node at (12.95,4.8){$y$};
\node[red] at (15,5.35){$C$};

\end{tikzpicture}
\caption{If two elements of $X$ assign different orientations to a particular $a_{n+1}$-cycle $C$ in $G_0$, then they must represent distinct subgroups of $\F_{n+1}$, as the word read upon traversing the path $\gamma$, then following the outgoing edge labeled with $a_{n+1}$, and then traversing the inverse of $\gamma'$ (left) cannot belong to both subgroups unless $a_{n+1}$ has order two.}
\end{figure}
\begin{remark}
Theorem~\ref{lab3} certainly applies to a large class of groups. Even so, the conditions of the theorem can be weakened. Indeed, the theorem holds whenever $G$ has a Cayley graph that contains infinitely many $a_i$-cycles (for some $i$) such that its fundamental group changes upon reversing the orientation of one (and hence any) such cycle. On the other hand, note that one cannot in general insist on a minimal generating set. This is impossible, for example, when $G$ is a free product of cyclic groups.
\end{remark}

To conclude, let us pose a concrete question to which we do not know the answer.

\begin{question}
Describe the invariant Schreier structures on $\Z^2$, the standard two-dimensional lattice.
\end{question}
It is not difficult to see that there exists a large number of invariant Schreier structures on $\Z^2$. Consider, as in the proof of Theorem~\ref{lab3}, the random Schreier structures one obtains by taking the standard Cayley structure on $\Z^2$ and randomly reversing the orientations of $a_i$-chains (that is, horizontal or vertical copies of $\Z$). Yet there are doubtless many more invariant Schreier structures, e.g.\ ones where $a_i$-cycles consist of ``infinite staircases,'' or finite-length cycles. It would be nice to have a full description of the geometric possibilities. Here is an even simpler question to which we do not know the answer:
\begin{question}
Describe the periodic Schreier structures on $\Z^2$.
\end{question}
By a periodic Schreier structure, we mean one whose orbit under the action of the free group is finite.

\end{document}